\documentclass{amsart}
\usepackage{amsmath, amssymb, amsthm}
\usepackage{url}
\usepackage{tikz}
\usepackage[utf8]{inputenc}
\usepackage[hidelinks]{hyperref}
\usepackage{mathtools}
\usepackage{pgfplots}
\usepackage{mathrsfs}
\usetikzlibrary{arrows}
\pgfplotsset{compat=1.15}
\usepackage{babel}

\numberwithin{equation}{section}
\theoremstyle{plain}
\newtheorem{theorem}{Theorem}[section]
\newtheorem{lemma}[theorem]{Lemma}

\newtheorem{corollary}[theorem]{Corollary}
\newtheorem{proposition}[theorem]{Proposition}
\newtheorem{example}[theorem]{Example}

\usepackage{etoolbox}

\begin{document}
	\title[A naive generalization of the hyperbolic and the quasihyperbolic metrics]{A naive generalization of the hyperbolic and the quasihyperbolic metrics}
	\date{\today}
	
	\author[B. Maji]{Bibekananda Maji}
	\address{Bibekananda Maji\\ Department of Mathematics \\
		Indian Institute of Technology Indore \\
		Simrol,  Indore,  Madhya Pradesh 453552, India.} 
	\email{bmaji@iiti.ac.in}
	
	\author[P. Naskar]{Pritam Naskar}
	\address{Pritam Naskar\\ Department of Mathematics \\
		Indian Institute of Technology Indore \\
		Simrol,  Indore,  Madhya Pradesh 453552, India.} 
	\email{naskar.pritam2000@gmail.com,   phd2201241008@iiti.ac.in }
	
	\author[S. K. Sahoo]{Swadesh Kumar Sahoo}
	\address{Swadesh Kumar Sahoo\\ Department of Mathematics \\
		Indian Institute of Technology Indore \\
		Simrol,  Indore,  Madhya Pradesh 453552, India.} 
	\email{swadesh.sahoo@iiti.ac.in}
	
	\maketitle

\begin{abstract}
Although the hyperbolic metric possesses many remarkable properties, it is not defined on arbitrary subdomains of $\mathbb{R}^n$ with $n \geq 2$. This article introduces a new hyperbolic-type metric that provides an alternative approach to this limitation. The proposed metric coincides with the hyperbolic metric on balls and half-spaces, and, quite unexpectedly, agrees with the quasihyperbolic metric in unbounded domains. We compute the density of this metric in several classical domains and discuss aspects of its curvature. Furthermore, we establish characterizations of uniform domains and John disks in terms of the newly defined metric. In addition, we investigate several geometric properties of the metric, including the existence of geodesics and the minimal length of non-trivial closed curves in multiply connected domains.
\end{abstract}

\maketitle

\section{\bf Introduction}\label{intro}
Hyperbolic geometry is one of the most iconic areas in the study of geometry and function theory. The hyperbolic metric was defined in the unit disk and the upper half-plane at its foundation. These two models were later shown to be conformally equivalent. By the Riemann Mapping Theorem, the hyperbolic metric can be extended to any simply connected planar domain, excluding the entire complex plane. Furthermore, using the universal covering map from the unit disk, the hyperbolic metric can be transferred to any hyperbolic planar domain, that is, any planar domain with at least two boundary points.

In $\mathbb{R}^n$, the hyperbolic metric is defined only in balls and half-spaces. However, unlike the planar case, it is not defined in more general subdomains of $\mathbb{R}^n$ for $n \geq 3$.

A significant development in this direction was the introduction of the quasihyperbolic metric by Gehring and Palka \cite{GP}. While this metric agrees with the hyperbolic metric in half-spaces, it does not coincide with it in balls. Therefore, identifying a general hyperbolic-type metric that aligns with the hyperbolic metric both in balls and half-spaces, or in even more general domains, is important. Notable examples addressing this challenge include the Apollonian metric \cite{Be}, the Ferrand metric \cite[p.~122]{Fe88}, and Seittenranta’s absolute ratio metric \cite{Se}, which are indeed Möbius-invariant and designed to serve this purpose.
Also, note that the Apollonian metric does not possess any geodesics \cite{HL}. In contrast, the Seittenranta and the Ferrand metrics possess geodesics, and they are nothing but circular arcs \cite{HIL}.

It is natural to seek a hyperbolic-type metric that coincides with the hyperbolic metric in balls and the quasihyperbolic metric in other domains. This motivates us to introduce a new metric that agrees with the hyperbolic metric in balls and half-spaces, and surprisingly, with the quasihyperbolic metric in unbounded domains. Moreover, we demonstrate that the new metric is bi-Lipschitz equivalent to the hyperbolic metric in all simply connected planar domains and to the quasihyperbolic metric in any bounded domain. Notably, the new metric is also geodesic.

\section{\bf Preliminaries}\label{Prel}
Throughout this article, we assume that $D$ is a proper subdomain of $\mathbb{R}^n$, $n\geq 2$, unless otherwise specified, and $|.|$ denotes the Euclidean norm. The diameter $d(D)$ of $D$ is defined by $d(D):=\sup\left\lbrace |x-y|:x,\,y\in D\right\rbrace$, which is considered as infinity for unbounded domains. For a point $z\in D$, the distance of $z$ from the boundary $\partial D$ of $D$ is defined by $\delta_D(z):=\inf\left\lbrace|z-\xi|:\xi\in \partial D \right\rbrace$. Let $\gamma$ be a rectifiable path joining $z_1$ and $z_2$ in $D$. We introduce a new length formula, namely, the $m_D$-length, which is defined by
	\begin{align}
		m_D(\gamma):=\int_{\gamma}\frac{d(D)\,|dz|}{\delta_D(z)\left(d(D)-\delta_D(z)\right)}.
	\end{align} 
We denote the density function $d(D)/[\delta_D(z)(d(D)-\delta_D(z))]$ by $m_D(z)$.
Using the above length formula, one can define a distance between two points $z_1$ and $z_2$ in the following way: 
	\begin{equation}\label{mDdefn}
			m_{D}\left(z_1,z_2\right)=\inf_{\gamma\in \Gamma_{z_1z_2}}	m_D(\gamma)=\inf_{\gamma\in \Gamma_{z_1z_2}}\int_{\gamma}m_D(z)\,|dz|,
	\end{equation}
where $\Gamma_{z_1z_2}$ denotes the set of all rectifiable paths $\gamma$ joining $z_1$ to $z_2$ in $D$. 
Note that the $m_D$-metric matches with the hyperbolic metric $h_B(z)\,|dz|$ in any ball $B=B(z_0,r)$, centered at $z_0$ and radius $r$, where $h_B(z)$ is the hyperbolic density of $B$ defined by
\begin{equation*}
	h_B(z)=\frac{2r}{r^2-|z-z_0|^2},
\end{equation*}
and the hyperbolic metric itself is defined by 
\begin{align}
	h_B(z_1,z_2):=\inf_{\gamma\in \Gamma_{z_1z_2}}\int_{\gamma}h_B(z)\,|dz|.
\end{align} 
In particular, when $\mathbb{D}:=B(0,1)\subset\mathbb{C}$, the above formula gives the standard Poincar\'{e} model of the hyperbolic metric on the unit disk. The hyperbolic metric can be defined not only for any simply connected proper subdomains of $\mathbb{C}$ (due to the Riemann Mapping Theorem (cf. \cite[p.~24]{BM})), but also for a planar domain $X$ with at least two boundary points (due to the uniformization theorem (cf. \cite[Chapter 7]{KeLa})).
The requirements of this article demand the definition of the quasihyperbolic metric:
\begin{equation}\label{kmetric}
	k_D(z_1,z_2):=\smashoperator{\inf_{\gamma\in \Gamma_{z_1z_2}}}\hspace{0.3 cm}\int_{\gamma}\frac{|dz|}{\delta_D(z)}.
\end{equation}
The reason why we stick with a bounded domain $D$ can easily be seen by calculating the limit $\lim_{d(D)\rightarrow \infty}m_D(z)|dz|$, which gives $|dz|/\delta_D(z)$. That is, in any unbounded domain, the newly defined metric $m_D$ is the same as the quasihyperbolic metric defined by \eqref{kmetric}.
Thus, the new metric $m_D$ generalizes both the hyperbolic and the quasihyperbolic metrics in a natural way.

Now, we recall some relevant, well-known results and definitions from the literature. A bounded simply connected domain $D\subset\mathbb{C}$ is said to be a $c$-John disk if there exists a constant $c\ge 1$ such that each pair of points $z_1$ and $z_2$ can be joined by a rectifiable arc $\gamma\in D$ for which 
\begin{equation}\label{doublecone condtn}
	\smashoperator{\min_{j=1,\,2}}\ell\left(\gamma\left(z_j,z\right)\right)\leq c\,\delta_D(z)
\end{equation}
holds for all $z\in \gamma$. This is also known as the double $c$-cone condition (or, twisted cone condition), and $\gamma$ is known as a double $c$-cone arc. A domain $D$ is said to be uniform if it satisfies (\ref{doublecone condtn}) and the quasi-convexity condition
\begin{align}
	\ell(\gamma)\leq c\,|z_1-z_2|.
\end{align}
It is a well known \cite{GO, Vu} result that a domain is uniform if and only if $k_D(z_1,z_2)\leq a\,j_D(z_1,z_2)$, for all $z_1,z_2\in D$ and for some constants $a\ge 1$, where the distance ratio metric $j_D$ is defined as follows:
\begin{align*}
	j_D\left(z_1,z_2\right)&=\frac{1}{2}\log\left\{\left(1+\frac{|z_1-z_2|}{\delta_D\left(z_1\right)}\right) \left(1+\frac{|z_1-z_2|}{\delta_D\left(z_2\right)}\right)\right\}.
\end{align*}
Since a John disk is more flexible than a uniform domain, one can ask whether there is any characterization of a John disk in terms of a similar metric inequality as in the case of uniform domains. The answer is given by Kim and Langmeyer in \cite{KL} that a domain $D$ is a John disk if and only if for any pair of points $z_1,\,z_2$ in $D$, the condition $h_D(z_1,z_2)\leq c\,j'_D(z_1,z_2)$ holds with some constant $c\geq 1$, where $j'_D$ is defined by
\begin{align*}
	j'_D\left(z_1,z_2\right)&=\cfrac{1}{2}\log\left\{\left(1+\cfrac{\lambda_D\left(z_1,z_2\right)}{\delta_D\left(z_1\right)}\right) \left(1+\cfrac{\lambda_D\left(z_1,z_2\right)}{\delta_D\left(z_2\right)}\right)\right\}.
\end{align*}
Note that $j'_D$ is a generalization of the distance ratio metric in the sense of replacing the Euclidean modulus by the inner distance $\lambda_D$ of $D$,  defined by
\begin{equation*}
	\lambda_D(z_1,z_2):=\inf_{\gamma\in \Gamma_{z_1z_2}}\ell(\gamma).
\end{equation*}
Also, the definition of a quasiconformal map is worth mentioning here. Let $D\subset\mathbb{R}^n$ be a domain and $f:D\rightarrow f(D)\subset\mathbb{R}^n$ be a homeomorphism. The linear dilatation of $f$ at $x\in D$ is defined by
\begin{align*}
	H(f,x):=\limsup_{r\rightarrow 0}\cfrac{\sup\left\{ |f(x)-f(y)|:|x-y|=r\right\}}{\inf\left\{ |f(x)-f(z)|:|x-z|=r\right\}}.
\end{align*}
If $H(f,x)$ is bounded by $K$, then the map $f$ is called a $K$-quasiconformal map. It is appropriate to remark that, due to Heinonen and Koskela \cite[Theorem 1.4]{HK}, the $\limsup$ in the definition of quasiconformality can be replaced with $\liminf$, and hence with only the limit.

The rest of the paper goes as follows: in Section \ref{prlm}, we first calculate the density of the metric in the annulus and the punctured unit ball and compare them with the hyperbolic metric. Then the monotonicity property and the equivalence with the quasihyperbolic metric have been studied, which will be helpful throughout the paper. Immediately, we deduce some characterizations of uniform domains and John disks in terms of the newly defined metric. Later, the quasi-invariance property of the metric has been studied under some special classes of functions, including conformal maps, in particular M\"{o}bius transformations, and quasiconformal maps. Then we studied the curvature property of the metric in the annulus and the punctured disk. Section \ref{gdcs} discusses the existence of the $m_D$-geodesics in bounded subdomains $D$ of $\mathbb{R}^n$, $n\geq 2$. In Section \ref{mltpl cnctd dmn & cmprsn}, a problem (suggested by Prof. T. Sugawa\footnote{The authors are thankful to Prof. T. Sugawa for suggesting this particular problem and bringing the article \cite{Got} to our attention during his visit to the ICSAMY conference at IIT Indore.}) related to the minimal length of a curve, which is not homotopic to a point, in a multiply connected planar domain has been studied.

Finally, the article ends with Section~\ref{Concluding remarks} containing some remarks and unsolved problems.

\section{\bf Basic examples and properties}\label{prlm}

\subsection{Comparison property}\label{example}
Here we will study the densities of the $m_D$-metric in some particular domains and compare them with the hyperbolic density.	
\begin{example}\label{ex1}
	\normalfont
Let $R>1$ and $A_{r,R}=\left\{z\in\mathbb{C}:0<r<|z|<R\right\}$. We are attempting to find the exact density function of the metric $m_{A_{r,R}}$.
\end{example}
For any $z\in A_{r,R}$, we have
\begin{align*}
	\delta_{A_{r,R}}(z)=\begin{cases}
		|z|-r, & \text{if $r<|z|\leq \cfrac{r+R}{2}$},\\
		R-|z|, & \text{if $\cfrac{r+R}{2}\leq|z|<R$}.
	\end{cases}
\end{align*}
Hence, the density function will be
\begin{align}
	m_{A_{r,R}}(z)=\begin{cases}\label{m-anulus}
		\cfrac{2R}{(|z|-r)(2R+r-|z|)}, & \text{if $r<|z|\leq \cfrac{r+R}{2}$},\\
		\hspace{0.8 cm}	\cfrac{2R}{R^2-|z|^2}, & \text{if $\cfrac{r+R}{2}\leq|z|<R$}.
	\end{cases}
\end{align}

\begin{corollary}\label{anulusmD}
	Let us denote $A_R:=A_{\frac{1}{R},R}$. Then the density is as follows:
	\begin{align*}
		m_{A_R}(z)=\begin{cases}
			\cfrac{2R^3}{(R|z|-1)(2R^2-R|z|+1)}, & \text{if $\cfrac{1}{R}<|z|\leq \cfrac{1+R^2}{2R}$},\\
			\hspace{0.8 cm}	\cfrac{2R}{R^2-|z|^2}, & \text{if $\cfrac{1+R^2}{2R}\leq|z|<R$}.
		\end{cases}
	\end{align*}
\end{corollary}
The hyperbolic density of $A_R$ can be found using the holomorphic universal covering map $e^z$ from the vertical strip of width $2\log R$ and symmetric about the real axis \cite{BM}. Indeed, the hyperbolic density is given by
\begin{align*}
	h_{A_R}(z)=\cfrac{\pi}{2\log R}\cfrac{1}{|z|\cos\left(\cfrac{\pi\log|z|}{2\log R}\right)}.
\end{align*}
Comparing these two densities in $A_R$ is not trivial, but it can be done intuitively using the graph (see Figure \ref{A_R fig}) of the functions
\begin{align*}
	f_1(x)=\begin{cases}
		\cfrac{2R^3}{(Rx-1)(2R^2-Rx+1)}, & \text{if $\cfrac{1}{R}<x\leq \cfrac{1+R^2}{2R}$},\\
		\hspace{0.8 cm}	\cfrac{2R}{R^2-x^2}, & \text{if $\cfrac{1+R^2}{2R}\leq x<R$},
	\end{cases}
\end{align*} 
and 
\begin{align*}
	g_1(x)=\cfrac{\pi}{2\log R}\cfrac{1}{x\cos\left(\cfrac{\pi\log x}{2\log R}\right)}.
\end{align*}

\begin{figure}[b]
	\centering
	\includegraphics[scale=1]{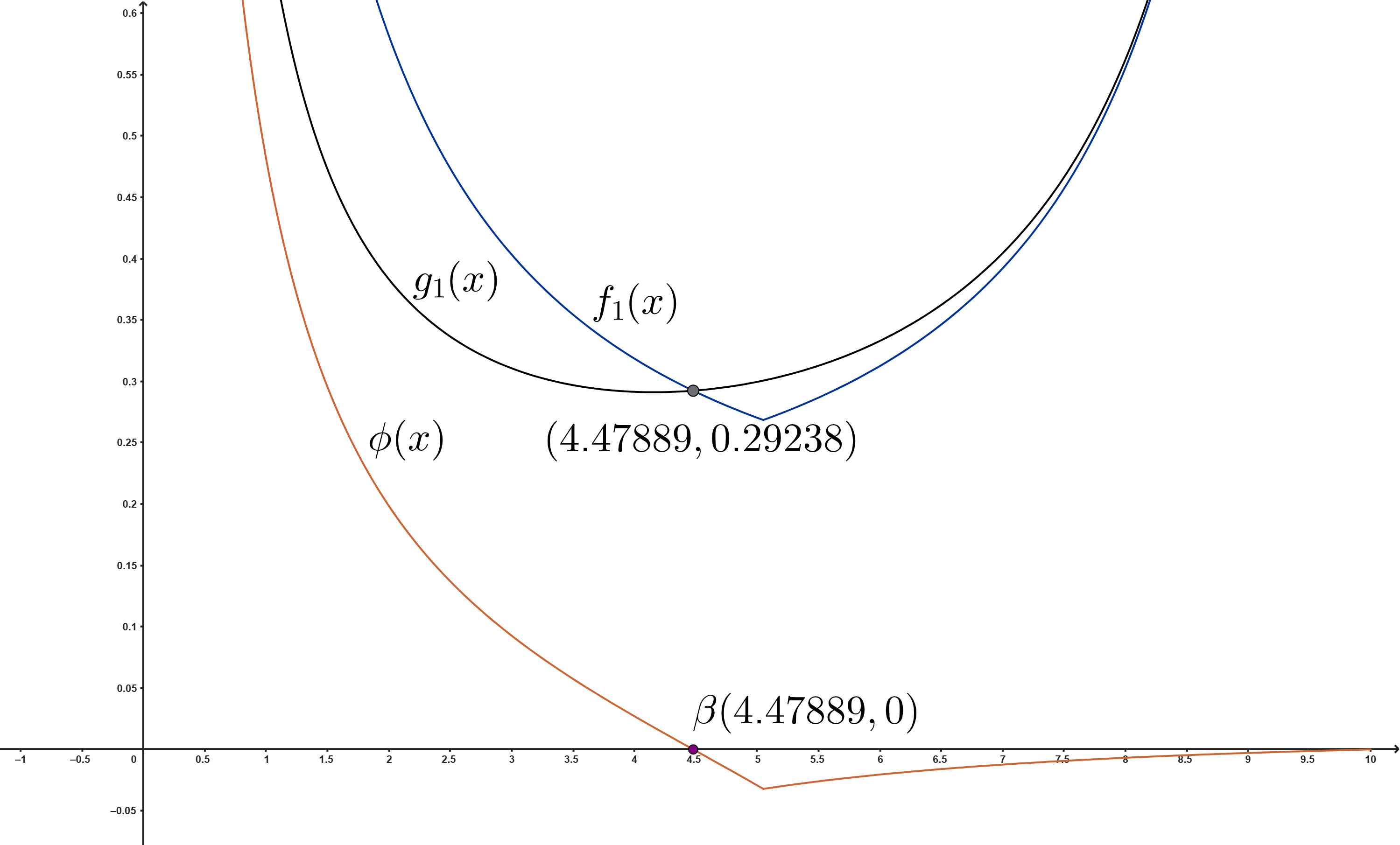}
	\caption{Graphs of $f_1(x)$, $g_1(x)$ and $\phi(x)$ with $R=10$.}
	\label{A_R fig}
\end{figure}
Here, we use a graphing calculator to visualize the densities of two metrics. Indeed, we have
\begin{align*}
	m_{A_R}(z)&\geq h_{A_R}(z), \hspace{0.2 cm}\text{for}\hspace{0.2 cm}  \cfrac{1}{R}<|z|\leq \beta,\\
	&\leq h_{A_R}(z), \hspace{0.2 cm}\text{for}\hspace{0.2 cm}  \beta\leq|z|<R,
\end{align*} 
where $\beta$ is the only positive root of the function $\phi(x)=f_1(x)-g_1(x)$ in $1/R<x<R$. 

\begin{example}\label{ex2}
	\normalfont
	For $R>0$, we set $\mathbb{D}_R:=\left\{z\in\mathbb{C}:|z|<R\right\}$. 
	Taking $r\rightarrow 0$ in (\ref{m-anulus}), the $m_D$-density for the punctured disk $\mathbb{D}_R^*=\mathbb{D}_R\setminus\left\{0\right\}$, can be obtained as
	\begin{align}\label{pncrd diks R}
		m_{\mathbb{D}_R^*}(z)=\begin{cases}
			\cfrac{2R}{|z|(2R-|z|)}, & \text{if $0<|z|\leq \cfrac{R}{2}$},\\
			\hspace{0.2 cm}	\cfrac{2R}{R^2-|z|^2}, & \text{if $\cfrac{R}{2}\leq|z|<R$}.
		\end{cases}
	\end{align}
In particular, if we take $R=1$, then the $m_D$-density in the punctured unit disk	$\mathbb{D}^*=\mathbb{D}\setminus\left\{0\right\}$ becomes
	\begin{align*}
	m_{\mathbb{D}^*}(z)=\begin{cases}
		\cfrac{2}{|z|(2-|z|)}, & \text{if $0<|z|\leq \cfrac{1}{2}$},\\
		\hspace{0.2 cm}	\cfrac{2}{1-|z|^2}, & \text{if $\cfrac{1}{2}\leq|z|<1$}.
	\end{cases}
\end{align*}
\end{example}
The hyperbolic density for $\mathbb{D}^*$ is known by the formula (see \cite[(7.17)]{KeLa})
	\begin{align*}
		h_{\mathbb{D}^*}(z)=\cfrac{1}{|z|\log(1/|z|)}.
	\end{align*}
	Considering the functions 
	\begin{align*}
		f_2(x)=\begin{cases}
		\cfrac{2}{x(2-x)},  &\text{if}  \ 0<x\leq \cfrac{1}{2},\\
		\cfrac{2}{1-x^2},  &\text{if}  \ \cfrac{1}{2}\leq x<1.  
		\end{cases}
	\end{align*}
	and $g_2(x)=\cfrac{1}{x\log(1/x)}$ and observing their graphs (Figure \ref{D_R^*fig}) one can conclude that 
	\begin{align*}
		m_{\mathbb{D}^*}(z)&\geq h_{\mathbb{D}^*}(z),\hspace{0.2 cm}\text{for}\hspace{0.2 cm}  0<|z|\leq \alpha,\\
		&\leq h_{\mathbb{D}^*}(z),\hspace{0.2 cm}\text{for}\hspace{0.2 cm}  \alpha\leq|z|< 1,
	\end{align*}
	
	\begin{figure}[h]
		\centering
		\includegraphics[scale=10]{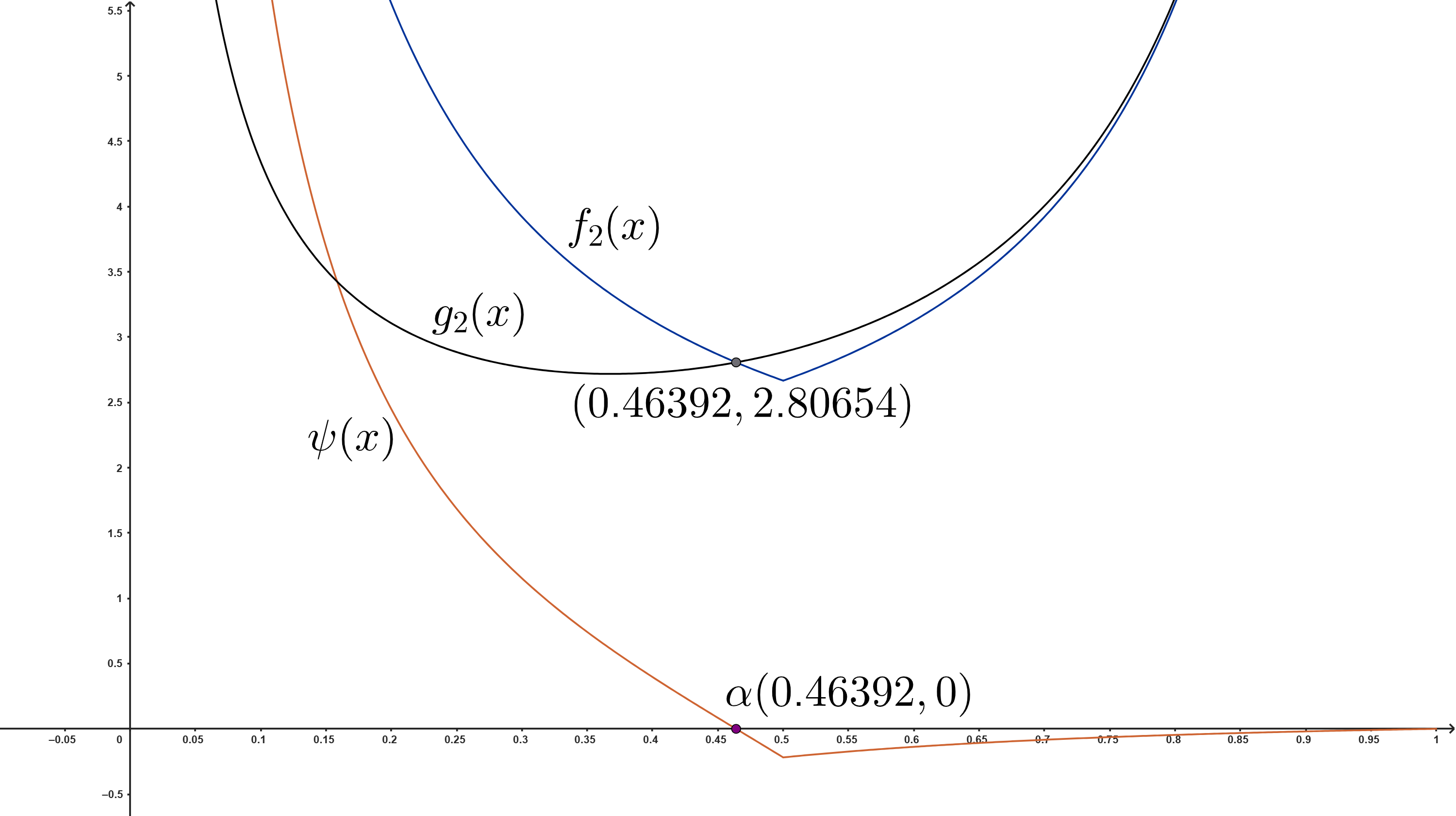}
		\caption{Graphs of $f_2(x)$, $g_2(x)$ and $\psi(x)$.}
		\label{D_R^*fig}
	\end{figure}
	where $\alpha$ is the root of the real function $\psi(x)=f_2(x)-g_2(x)$ in $0<x<1$.	
	
\subsection{Monotonicity and bi-Lipschitz properties}
For comparison of hyperbolic-type metrics, the domain monotonicity property plays a crucial role. Thus, we study the domain monotonicity property of the $m_D$-metric.
\begin{lemma}\label{monotncty}
Let $D_1$ and $D_2$ be two bounded proper subdomains of $\mathbb{R}^n$. If $D_1\subset D_2$ then $m_{D_2}(z)\leq m_{D_1}(z)$ on $D_1$, where $m_D(z)$ denotes the density of the $m_D$-metric in a bounded domain $D\subset \mathbb{C}$.
\end{lemma}

\begin{proof}
	We use basic properties from calculus to prove this result. For a fixed $y>0$, consider the function
	$$f(x)=\cfrac{y}{x(y-x)},\, x\in\left(0,\infty\right).$$
	Then $$f'(x)=\cfrac{y(2x-y)}{x^2(y-x)^2}.$$
	If the function $f$ is strictly increasing then we must have $2x>y.$ But the choice $y=d(D_2)$ and $x=\delta_{D_2}(z)$, leads to the contradiction that $2\delta_{D_2}(z)>d(D_2).$ So we must conclude that $f$ is a monotone decreasing function. Now, $D_1\subset D_2$ implies that $x=\delta_{D_2}(z)\geq\delta_{D_1}(z)=p$, say. Then we have $f(x)\leq f(p)=\cfrac{y}{p(y-p)}=:g(y)$. Again $$g'(y)=-\cfrac{1}{(y-p)^2}<0.
	$$
	Thus $g$ is also decreasing. So,  $y=d(D_2)\geq d(D_1)=q$, implies $g(y)<g(q)$. Hence 
	\begin{align*}
		f(x)\leq g(y)<g(q)
		\implies& \cfrac{y}{x(y-x)}<\cfrac{q}{p(q-p)}\\
		\implies &\cfrac{d(D_2)}{\delta_{D_2}(z)(d(D_2)-\delta_{D_2}(z))}<\cfrac{d(D_1)}{\delta_{D_1}(z)(d(D_1)-\delta_{D_1}(z))},
	\end{align*}
	which establishes the domain monotonicity property of the $m_D$-metric.
\end{proof}

\begin{theorem}\label{equiv k&m}
	Let $D\subsetneq\mathbb{R}^n$, $n\geq2$, be a bounded domain. Then for any $z\in D$,
	\begin{equation}
		\cfrac{1}{\delta_D(z)}\leq m_{D}(z)\leq \cfrac{2}{\delta_D(z)}.
	\end{equation} 
	That is, the metric $m_D$ is bi-Lipschitz equivalent to the quasihyperbolic metric $k_D$.
\end{theorem}

\begin{proof}
	Let $z_0\in D$ be arbitrarily fixed and $r=\delta_D(z_0)$. For the upper bound, let us consider the disk $B=B\left(z_0,r\right)\subseteq D$. Then, by Lemma \ref{monotncty}, 
	\begin{equation*}
		m_D(z)\leq m_{B}(z)=\cfrac{2r}{r^2-|z-z_0|^2}.
	\end{equation*}
	In particular, 
	\begin{equation*}
		m_D(z_0)\leq m_{B}(z_0)=\cfrac{2}{\delta_D(z_0)}.
	\end{equation*}
	Since the point $z_0$ was chosen arbitrarily, the upper bound follows. 
	The constant $2$ is the best possible, which can be obtained by taking $D$ to be a disk with center at $z$.
	
	To prove the lower bound, for some $c>0$, consider the function $f:\left(0,c/2\right]\rightarrow\left(1,2\right]$ defined by $$f(x)=\cfrac{c}{c-x}.$$
Since $f'(x)>0$ for all $x\in \left(0,c/2\right]$, $f$ is strictly increasing in $\left(0,c/2\right]$. As $x\to 0^+$, $f(x)\to 1$.  Hence, $f(x)> 1$ for all $x\in \left(0,c\right)$. Taking $c=d(D)$ and $x=\delta_D(z)$, we have the following inequality:
	\begin{align*}
		\cfrac{d(D)}{d(D)-\delta_D(z)}>1
		\implies \,m_D(z)>\cfrac{1}{\delta_D(z)}.
	\end{align*}
	For $D=B(0,1)$, it follows that
	\begin{align*}
		m_D(z)\delta_D(z)=\cfrac{2}{1-|z|^2}\left(1-|z|\right)=\cfrac{2}{1+|z|}\rightarrow 1 \,\,\text{as}\,\, |z|\rightarrow 1.
	\end{align*}
	Hence, the lower bound is concluded, and it is sharp.
	\end{proof}

\subsection{Characterization of domains} 
There are significant implications of Theorem \ref{equiv k&m} related to the characterization of some special domains, namely, uniform domains and John disks. The first result gives a characterization of uniform domains in terms of the new metric and the distance ratio metric.
\begin{proposition}
	Let $D$ be a domain in $\mathbb{R}^n,\, n\geq 2$. Then $D$ is a uniform if and only if there exists a constant $c_1\ge 1$ such that any two points $x,\,y$ in $D$
	\begin{align}
		m_D(x,y)\leq c_1\,j_D(x,y).
	\end{align}
\end{proposition}
 The proof follows directly from Theorem \ref{equiv k&m} and the characterization of uniform domains due to Gehring and Osgood \cite{GO}. 
 The next result is related to the characterization of John disks.
\begin{proposition}\label{smplycnctdcse}
	Let $D$ be a simply connected bounded domain in $\mathbb{C}$. Then $D$ is a John disk if and only if there exists a constant $c_2\ge 1$ such that any two points $z_1,\,z_2$ in $D$
	\begin{align}\label{rerd condn}
		m_D(z_1,z_2)\leq c_2\,j'_D(z_1,z_2).
	\end{align}
\end{proposition}
One can see the following example to observe the importance of Proposition \ref{smplycnctdcse}.
\begin{example}
	Let $R=\left\{x+iy:-1<x<0, -1<y<0\right\}$ and $D=R\setminus\left(\mathbb{D}\cap\left\{x+iy:x,\,y<0\right\}\right)$.
\end{example}	
It has been shown in \cite{HPSW} that $D$ is not a John disk. Indeed, we can also show that by disproving (\ref{rerd condn}). Suppose, on the contrary, that for any two points $z_1,z_2$ in $D$, (\ref{rerd condn}) holds. Note that the hyperbolic metric in $D$ is equivalent to the $m_D$-metric due to the equivalence of $h_D$ and $k_D$ in any simply connected domain $D$ (cf. \cite[p.~36]{BM}) and Theorem \ref{equiv k&m}. Thus, we have $h_D(z_1,z_2)\leq c\,j'_D(z_1,z_2)$. Hence, by \cite[Theorem 4.1]{KL}, $D$ is a John disk, a contradiction.

\subsection{Quasi-invariance properties under special maps}\label{qsi-inv}
We are now interested to see certain quasi-invariance properties of the $m_{D}$-metric under some important classes of functions. It is well-known that M\"{o}bius transformations are hyperbolic isometries. Thus, if $f:\mathbb{D}\rightarrow \mathbb{D}$ is a M\"{o}bius transformation then for any $z_1,z_2\in \mathbb{D}$, $$h_{\mathbb{D}}(z_1,z_2)=h_{\mathbb{D}}(f(z_1),f(z_2)).$$
The isometries of the quasihyperbolic metric have been studied by H\"{a}st\"{o} \cite{Ha07}. The quasi-invariance properties of the quasihyperbolic metric under a conformal map have been studied in \cite{KVZ}. The following is true for the $m_D$-metric, which follows from Theorem \ref{equiv k&m} and \cite[Proposition 1.6]{KVZ}. Note that the inequality remains true even if the image domain $D':=f(D)$ is unbounded, since $m_{D'}=k_{D'}$ in this case.
\begin{proposition}\label{conf-inv}
	Let $D\subsetneq\mathbb{C}$ be a bounded domain and $f$ map conformally onto $D'=f(D)$. Then
	\begin{align*}
		\cfrac{1}{8}\,m_D(x,y)\leq m_{D'}(f(x),f(y))\leq 8\,m_D(x,y)	,	
	\end{align*}
	for all $x,\,y\in D$.
\end{proposition}

It is expected that the bi-Lipschitz constant $8$ of Proposition~\ref{conf-inv} can be improved by considering a special class of conformal maps, viz. M\"{o}bius transformation. For this, let us first recall \cite[Lemma 2.4]{GP} due to  Gehring and Palka.
\begin{lemma}\label{k and mob}
	Let $D_1$ and $D_2$ be two proper subdomains of $\mathbb{R}^n,\, n\geq 2$. If $f:D_1\rightarrow D_2$ is a M\"{o}bius transformation then
	\begin{equation}
		\cfrac{1}{\delta_{D_1}(z)}\leq 2\,|f'(z)|\,\cfrac{1}{\delta_{D_2}(f(z))}
	\end{equation}
	for all $z\in D_1$.
\end{lemma}
By Theorem \ref{equiv k&m} and Lemma \ref{k and mob}, we have the following result.
\begin{corollary}
	Let $D_1, D_2$ be proper bounded subdomains of $\mathbb{R}^n$, $n\geq 2$. If $f:D_1\rightarrow D_2$ is a M\"{o}bius transformation then
	\begin{equation*}\label{mobqinv}
		m_{D_1}(z)\leq 4\,|f'(z)|\,m_{D_2}(f(z))
	\end{equation*}
	for all $z\in D_1$.
\end{corollary}
\begin{proposition}\label{mobinvmd}
	Let $D_1$ and $D_2$ be two proper subdomains of $\mathbb{R}^n,\, n\geq 2$. If $f:D_1\rightarrow D_2$ is a M\"{o}bius transformation, then we have
	\begin{equation}
		\frac{1}{4}\,m_{D_1}(z_1,z_2)\leq m_{D_2}(f(z_1),f(z_2))\leq 4\,m_{D_1}(z_1,z_2)
	\end{equation}
	for all $z_1,z_2\in D_1$.
\end{proposition}
\begin{proof}
	Let $\gamma'$ be a rectifiable path joining $f(z_1)$ and $f(z_2)$ in $D_2$. Suppose $\gamma=f^{-1}(\gamma')$. Then 
	\begin{align*}
		m_{D_1}(z_1,z_2)&\leq \int_{\gamma}m_{D_1}(z) \,|dz|\\
		&\leq 4\int_{\gamma} |f'(z)|m_{D_2}(f(z))\, |dz|\\
		&=4\int_{\gamma'} m_{D_2}(z) \,|dz|,
	\end{align*}
where the second inequality follows from Corollary \ref{mobqinv}. Finally, taking infimum over all paths $\gamma'$ will give us the first half of the inequality. The upper bound can be shown in a similar fashion by changing the role of $D_1$ and $D_2$.
\end{proof}
Note that, since quasihyperbolic metric is $M^2$-quasi-invariant under an $M$-bi-Lipschitz map (\cite{KVZ}, Lemma 1.3), the metric $m_D$ is also quasi-invariant under an $M$-bi-Lipschitz map with bi-Lipschitz constant $2M^2$. Also, the following quasi-invariance property of the $m_D$-metric under quasiconformal mappings is an immediate consequence of Theorem \ref{equiv k&m} and \cite[Theorem 3]{GO}.

\begin{theorem}
	If $f$ is a $K$-quasiconformal mapping of $D_1$ to $D_2$, both are proper subdomain of $\mathbb{R}^n$, then there exists a constant $c$, depending only $n$ and $K$, with the following property:
	\begin{align*}
		m_{D_2}(f(x),f(y))\leq c\,\max\left\{m_{D_1}(x,y),m_{D_1}(x,y)^{\alpha}\right\}, \hspace{0.5 cm}\alpha=K^{1/(1-n)},	
	\end{align*}
	for all $x,\,y\in D$.
\end{theorem}

\subsection{Curvature}\label{curvature} It is quite a tedious task to calculate the Gaussian curvature of the metric in arbitrary domains. But we calculate them for some special bounded domains that we already considered in Section \ref{example}. Let us first recall the definition of the curvature of a conformal metric. For a positive $C^2$ function $\rho$ defined on an open subset $U$ of $\mathbb{C}$, the curvature of the metric $\rho(z)|dz|$ at a point $a\in U$ is given by 
\begin{align*}
	K_{\rho}(a)=-\cfrac{\Delta\log\rho(a)}{\rho(a)^2},
\end{align*} 
where $\Delta$ denotes the Laplacian,
\begin{align*}
	\Delta=\cfrac{\partial^2}{\partial x^2}+\cfrac{\partial^2}{\partial y^2}.
\end{align*}
For our purpose, it is convenient to use the complex form of the Laplacian $\Delta$ given by
\begin{align*}
	\Delta=4\, \cfrac{\partial^2}{\partial z\partial\bar{z}}=4\, \cfrac{\partial^2}{\partial \bar{z}\partial z}.
\end{align*}

\subsubsection{\bf Curvature of the $m_D$-metric in annulus}\label{crvtr of anls} First, we see the curvature of the metric in the annulus $A_R$, whose density is already calculated in Corollary \ref{anulusmD} and is given by
\begin{align*}
	m_{A_R}(z)=\begin{cases}
		\cfrac{2R^3}{(R|z|-1)(2R^2-R|z|+1)}, & \text{if $\cfrac{1}{R}<|z|\leq \cfrac{1+R^2}{2R}$},\\
		\hspace{0.8 cm}	\cfrac{2R}{R^2-|z|^2}, & \text{if $\cfrac{1+R^2}{2R}\leq|z|<R$}.
	\end{cases}
\end{align*}
For any point $z\in A_R$ satisfying $1/R<|z|< {(1+R^2)}/{2R}$, 
\begin{align*}
	\Delta \log m_{A_R}(z)
	&=4\,\cfrac{\partial^2}{\partial\bar{z}\partial z}\log m_{A_R}(z)\\
	&=-4\,\cfrac{\partial^2}{\partial\bar{z}\partial z}\left\{\log(R|z|-1)+\log(2R^2-R|z|+1)\right\}\\
	&=-4\,\left\{\cfrac{R}{2}\,z^{-1/2}\frac{\partial}{\partial\bar{z}}\left(\frac{\bar{z}^{1/2}}{R|z|-1}\right)-\frac{R}{2}\,z^{-1/2}\frac{\partial}{\partial\bar{z}}\left(\frac{\bar{z}^{1/2}}{2R^2-R|z|+1}\right)\right\}\\
	&=-4\,\left\{-\frac{R}{4}\frac{1}{|z|(R|z|-1)^2}-\frac{R}{4}\frac{2R^2+1}{(2R^2-R|z|+1)^2}\right\}\\
	&=\cfrac{2R}{|z|}\,\frac{R^4(2+|z|^2)-4R^3|z|+R^2(3+|z|^2)-2R|z|+1}{(R|z|-1)^2(2R^2-R|z|+1)^2},
\end{align*}
and hence 
\begin{align*}
	K_{m_{A_R}}(z)&=-\cfrac{\Delta \log m_{A_R}(z)}{m_{A_R}(z)^2}\\
	&=-\cfrac{1}{2R^5|z|}\left\{R^4(2+|z|^2)-4R^3|z|+R^2(3+|z|^2)-2R|z|+1\right\}.%\in \left(-\frac{R^6+11R^4-R^2-3}{8R^4},-\frac{1}{R^4}\right)
\end{align*}
The above expression lies in the open interval $\left({(-R^6-11R^4+R^2+3)}/{8R^4},-{1}/{R^4}\right)$. So, the metric has non-constant but bounded curvature at all the points $z\in A_R$ satisfying $1/R<|z|< {(1+R^2)}/{2R}$. For a point $z\in A_R$, satisfying $(1+R^2)/{2R}<|z|<R$, it can easily be seen that the curvature is $-1$. Note that, since the density of $m_{A_R}$-metric is not $C^2$ on the circle $|z|=(1+R^2)/{2R}$, the curvature is undetermined on this circle (see Figure~\ref{A_R fig} for $R=10$). 

\subsubsection{\bf Curvature of the $m_D$-metric in punctured disk} 
In Example \ref{ex2}, we have seen that
	\begin{align*}
		m_{\mathbb{D}_R^*}(z)=\begin{cases}
			\cfrac{2R}{|z|(2R-|z|)}, & \text{if $0<|z|\leq \cfrac{R}{2}$},\\
			\hspace{0.2 cm}	\cfrac{2R}{R^2-|z|^2}, & \text{if $\cfrac{R}{2}\leq|z|<R$}.
		\end{cases}
	\end{align*}
	For any point $z\in \mathbb{D}_R^*\cap\left\{0<|z|< \cfrac{R}{2}\right\}$, we first calculate
	\begin{align*}
		\Delta\log m_{\mathbb{D}_R^*}(z)
		=4\,\cfrac{\partial^2}{\partial\bar{z}\partial z}\log m_{\mathbb{D}_R^*}(z)
		&=4\,\cfrac{\partial^2}{\partial\bar{z}\partial z}\log\frac{2R}{|z|(2R-|z|)}\\
		&=-4\,\cfrac{\partial}{\partial\bar{z}}\left(\frac{Rz^{-1/2}\bar{z}^{1/2}-\bar{z}}{2R|z|-|z|^2}\right)\\
		&=-4\,\cfrac{(2R-|z|)(R-2|z|)-2(R-|z|)^2}{2(2R-|z|)^2}\\
		&=\cfrac{2R}{|z|(2R-|z|)^2}.
	\end{align*}
	Hence, the curvature at the point $z$ is
	\begin{align*}
		K_{m_{\mathbb{D}_R^*}}(z)=-\cfrac{\Delta\log m_{\mathbb{D}_R^*}(z)}{m_{\mathbb{D}_R^*}(z)^2}=-\cfrac{\cfrac{2R}{|z|(2R-|z|)^2}}{\left(\frac{2R}{|z|(2R-|z|)}\right)^2}=-\cfrac{|z|}{2R}\in\left(-\frac{1}{4},0\right).
	\end{align*}
It is easy to check that for any point $z\in \mathbb{D}_R^*\cap\left\{z:\cfrac{R}{2}<|z|<R\right\}$, the curvature
\begin{align*}
	K_{m_{\mathbb{D}_R^*}}(z)=-\cfrac{\Delta\log m_{\mathbb{D}_R^*}(z)}{m_{\mathbb{D}_R^*}(z)^2}=-1.
\end{align*}	
We notice that it is not possible to conclude the curvature on the circle $|z|=R/2$ due to the lack of $C^2$-differentiability of the metric $m_{\mathbb{D}_R^*}$ on this circle (see Figure~\ref{D_R^*fig} for $R=1$). On the region $\mathbb{D}_R^*\setminus\left\{z:|z|=R/2\right\}$ the metric has non-positive bounded curvature. The curvature is constant in the region $R/2<|z|<R$ and it is non-constant in the remaining region. However, the metric becomes flat near the origin.
 
\section{\bf Existence of geodesics}\label{gdcs} In this section, we show the existence of the $m_D$-geodesics, i.e., the existence of the rectifiable path joining two points for which the infimum attains in the definition \eqref{mDdefn}. Helly's theorem \cite[Corollary 14.25]{Ca} will play a crucial role in this proof. We begin with a lemma, due to Gehring and Palka \cite[Lemma 2.1]{GP}, which will help us to prove an important inequality for the $m_D$-metric.
\begin{lemma}
	If $D$ is a proper subdomain of $\mathbb{R}^n$, then 
	\begin{align}
			&k_D(x,y)\geq \log\cfrac{\delta_D(y)}{\delta_D(x)},\label{kloginqlty1}\\
			&k_D(x,y)\geq \log\left(1+\cfrac{|x-y|}{\delta_D(x)}\right),\label{kloginqlty2}
		\end{align}
	for all $x,y\in D$.
\end{lemma}
Let us note that, for any two points $z_1,z_2$ in a bounded domain $D$,
\begin{align*}
	\cfrac{d(D)}{\delta_D(z)(d(D)-\delta_D(z))}=\cfrac{1}{\delta_D(z)}+\cfrac{1}{d(D)-\delta_D(z)}
\end{align*} 
gives
\begin{align}\label{mkn}
	m_D(z_1,z_2)\geq k_D(z_1,z_2)+n_D(z_1,z_2),
\end{align}
with
\begin{equation}
	n_D(x,y):=\inf_{\gamma\in \Gamma{xy}}\int_{\gamma}\cfrac{|dz|}{d(D)-\delta_D(z)}.
\end{equation}
This also defines a metric in any bounded subdomain $D\subset\mathbb{R}^n$, $n\geq2$. In order to prove the existence of the geodesics of the $m_D$-metrics, we take the help of the $n_D$-metric and deduce an inequality for the $m_D$-metric.

\begin{lemma}
	Let $D$ be a bounded subdomain of $\mathbb{R}^n$, $n\geq2$. Then for any two points $x,\,y\in D$ we have
	\begin{equation}\label{nloginequality1}
		n_D(x,y)\geq \log\cfrac{d(D)-\delta_D(y)}{d(D)-\delta_D(x)}.
	\end{equation}
\end{lemma}
\begin{proof}
	Let us fix two points $x,\,y\in D$ arbitrarily and choose a rectifiable path $\gamma$ in $D$ which is parameterized over $\left[0,1\right]$ with $\gamma(0)=x$ and $\gamma(1)=y$. Set $\alpha=d(D)-\delta_D(x)$ and $\beta=d(D)-\delta_D(y)$. Without loss of generality, one can assume that $\alpha\leq\beta$. Let us define a function $f:[0,1]\to \mathbb{R}$ by 
$$
f(t)=d(D)-\delta_D(\gamma(t)).
$$ 
For a fixed positive integer $m$, set $t_j$ to be the first point in $\left[0,1\right]$ such that 
	\begin{align}\label{f(t)=c_j=alpha,beta}
	f(t_j)=c_j=\alpha\left(\cfrac{\beta}{\alpha}\right)^{j/m}.
    \end{align} 
	Then we obtain a partition of a segment of $\left[0,1\right]$, viz. $0=t_0<t_1<\dots<t_m\leq 1$ with the properties that $f(t_j)=c_j$ and $f(t)\leq c_j$, for any $t\in \left[0,t_j\right]$. Writing $\gamma_j:=\gamma|_{\left[t_{j-1},t_j\right]}$, we have
	\begin{align*}
    \int_{\gamma_j}n_D(z)\,|dz|=\int_{t_{j-1}}^{t_j}\cfrac{1}{f(t)}\,|d\gamma(t)|&\geq \cfrac{1}{c_j}\int_{t_{j-1}}^{t_j}|d\gamma(t)|\\
	&\geq\cfrac{1}{c_j}|\gamma(t_j)-\gamma(t_{j-1})|\\
	&\geq\cfrac{1}{c_j}\left||\gamma(t_j)-p_{j-1}|-\delta_D(\gamma\left(t_{j-1}\right))\right|.
	\end{align*}
   In the last inequality, we used the reversed triangle inequality, where $p_{j-1}\in\partial D$ with $\delta_D(\gamma(t_{j-1}))=|\gamma(t_{j-1})-p_{j-1}|$. Hence,
	\begin{align*}
		\int_{\gamma_j}n_D(z)\,|dz|&\geq\cfrac{1}{c_j}\,\left|\delta_D\left(\gamma\left(t_j\right)\right)-\delta_D\left(\gamma\left(t_{j-1}\right)\right)\right|\\
		&=\,\cfrac{1}{c_j}\left|\left\{d(D)-\delta_D\left(\gamma\left(t_j\right)\right)\right\}-\left\{d(D)-\delta_D\left(\gamma\left(t_{j-1}\right)\right)\right\}\right|\\
		&\geq\,\cfrac{1}{c_j}\left\{f\left(t_j\right)-f\left(t_{j-1}\right)\right\}\\
		&=\cfrac{c_j-c_{j-1}}{c_j}.
	\end{align*}
Now, utilizing \eqref{f(t)=c_j=alpha,beta}, we have
\begin{align*}
	\int_{\gamma_j}n_D(z)|dz|\geq \cfrac{c_j-c_{j-1}}{c_j}&=\left(\cfrac{\alpha}{\beta}\right)^{1/m}\left[\left(\cfrac{\beta}{\alpha}\right)^{1/m}-1\right]\\
	&\geq\left(\cfrac{\alpha}{\beta}\right)^{1/m}\log\left(\cfrac{\beta}{\alpha}\right)^{1/m}\\
	&=\left(\cfrac{\alpha}{\beta}\right)^{1/m}\log\left(\cfrac{c_j}{c_{j-1}}\right).
\end{align*}
Therefore, we obtain
\begin{align*}
	\int_{\gamma}n_D(z)|dz|\geq\sum_{j=1}^{m}\int_{\gamma_j}n_D(z)|dz|\geq\left(\cfrac{\alpha}{\beta}\right)^{1/m}\sum_{j=1}^{m}\log\left(\cfrac{c_j}{c_{j-1}}\right)=\left(\cfrac{\alpha}{\beta}\right)^{1/m}\log\left(\cfrac{\beta}{\alpha}\right),
\end{align*}
and letting $m\rightarrow\infty$, we see that
\begin{equation*}
	\int_{\gamma}n_D(z)|dz|\geq \log\left(\cfrac{d(D)-\delta_D(y)}{d(D)-\delta_D(x)}\right).
\end{equation*}
Finally, taking infimum over all curves joining $x,y$ in $D$ gives the required inequality.
\end{proof}
\begin{corollary}\label{mloginequalitycor}
		Let $D\subsetneq\mathbb{R}^n$ be a bounded domain. Then for any two points $x,\,y\in D$, we have
	\begin{equation}\label{mloginqlty1}
		m_D(x,y)\geq \log\left(\cfrac{\delta_D(y)\left(d(D)-\delta_D(y)\right)}{\delta_D(x)\left(d(D)-\delta_D(x)\right)}\right).
	\end{equation}
\end{corollary}
\begin{proof}
	It is easy to see that the use of (\ref{kloginqlty1}) and (\ref{nloginequality1}) in (\ref{mkn}) gives this crude but useful lower bound for the $m_D$-metric.
\end{proof}
Finally, we gathered all the necessary tools for the proof of the existence of $m_D$-geodesics.
\begin{theorem}\label{mgeodscs}
	Let $D$ be a bounded subdomain of $\mathbb{R}^n$ and $x_1,x_2\in D$. Then there always exists an $ m_D$-geodesic $\gamma$ joining them in $D$.
\end{theorem}

\begin{proof}
	Let us arbitrarily fix two points $x_1,x_2\in D$. Considering the definition of $m_D$-metric and the infimum property, we see that there exists a sequence of rectifiable arcs $\left\{\gamma_i\right\}\subset \Gamma_{x_1x_2}$ such that
\begin{equation*}
		m_{D}\left(x_1,x_2\right)=\lim_{i\rightarrow\infty} \int_{\gamma_i}m_D(z)\,|dz|.
\end{equation*}	
Note that each $\gamma_i$ is arc length parametrized on the same interval $\left[0,1\right]$. Since the limit of the sequence exists, it is bounded, and we denote it by
\begin{equation*}
	\eta:=\sup_{i}\int_{\gamma_i}m_D(z)\,|dz|<\infty.
\end{equation*}
If $x\in \gamma_j$, for some $j$, then Corollary \ref{mloginequalitycor} gives
\begin{equation*}
	\log\left(\cfrac{\delta_D(x)\left(d(D)-\delta_D(x)\right)}{\delta_D(x_1)\left(d(D)-\delta_D(x_1)\right)}\right)\leq m_D(x_1,x)\leq\int_{\gamma_i}m_D(z)\,|dz|\leq\eta,
\end{equation*}
which implies
\begin{align*}
	\delta_D(x)\left(d(D)-\delta_D(x)\right)\leq e^{\eta}\,\delta_D(x_1)\left(d(D)-\delta_D(x_1)\right).
\end{align*}
Therefore, for each $\gamma_j$, we see that
\begin{align*}
	\ell\left({\gamma_j}\right)=\int_{\gamma_j}|dz|&\leq \cfrac{e^{\eta}}{d(D)}\,\delta_D(x_1)\left(d(D)-\delta_D(x_1)\right)\int_{\gamma_j}m_D(z)\,|dz|\\
	&\leq\cfrac{\eta e^{\eta}}{d(D)}\,\delta_D(x_1)\left(d(D)-\delta_D(x_1)\right).
\end{align*}
This shows that the Euclidean length of $\gamma_j$ is uniformly bounded. Hence, the sequence of infinitesimal length elements is of bounded variation, and the sequence of total variations $\left\{V(|dz|)\right\}_j=\left\{\ell\left({\gamma_j}\right)\right\}$ forms a uniformly bounded sequence. Then the use of Helly's theorem \cite[Corollary 14.25]{Ca} gives the existence of a subsequence $\left\{j_k\right\}$ of the indexing set such that $\gamma_{j_k}$ converges pointwise to a rectifiable path $\gamma$ joining $x_1,x_2$ in $D$ with the following:
\begin{align*}
	m_D(x_1,x_2)=\lim_{k\rightarrow\infty}\int_{\gamma_{j_k}}m_D(z)\,|dz|=\int_{\gamma}m_D(z)\,|dz|.
\end{align*}
This completes the proof.
\end{proof}
 
\section{\bf The $m_D$-length in multiply connected domains}\label{mltpl cnctd dmn & cmprsn}
In this section, we consider a multiply connected proper subdomain $D$ of $\mathbb{C}$. A closed curve $\gamma\subset D$ is called non-trivial if it is not homotopic to a point. We are interested in finding the least $m_D$-length of such curves. The proposed problem is equivalent to finding the best possible constant $C$ satisfying
\begin{align*}
	\int_{\gamma}m_D(z)|dz|\geq C,
\end{align*}
 for all non-trivial closed curves $\gamma$ in $D$. A similar study for the Hahn metric can be found in \cite{Got}. For any multiply connected planar domain $D$, there always exists a punctured plane $\mathbb{C}^*_a:=\mathbb{C}\setminus \left\{a\right\}$ containing $D$. Indeed, we choose the point $a$ to be any point situated inside one of the holes of $D$ surrounded by the path $\gamma$. The above discussion furnishes the following proposition.
 
\begin{proposition}
Let $\gamma$ be a non-trivial closed curve in a multiply connected planar domain $D$. Then
\begin{align}
\int_{\gamma}m_D(z)|dz|\geq 2\pi.
\end{align}
The equality holds for 
$D=\mathbb{D}_R^*$ and $\gamma=\left\{z:|z|=r,~ 0<r\le R/2\right\}$ with either $R$ arbitrarily large or $r$ arbitrarily small.
\end{proposition}
\begin{proof}
Without loss of generality, we assume that any one of the holes of $D$ surrounded by the non-trivial curve $\gamma$ contains the origin. Then the punctured plane $\mathbb{C}^*$ will be the largest planar multiply connected domain containing $D$. Thus, by Lemma \ref{monotncty}, we have 
 	\begin{align*}
 		k_{\mathbb{C}^*}(z)=m_{\mathbb{C}^*}(z)\leq m_D(z)
 	\end{align*}
 	for all $z$ in $D$ and the inequality is sharp. Therefore, we obtain
 	\begin{align*}
 		\int_{\gamma}m_D(z)|dz|\geq \int_{\gamma}k_{\mathbb{C}^*}(z)|dz|
 		=\int_{\gamma}\cfrac{|dz|}{|z|}\geq \int_{|z|=s}\cfrac{|dz|}{|z|}=2\pi,
 	\end{align*}
 	where $s$ is the radius of the smallest circle surrounding $\gamma$. The sharpness is clear from the last step of the above calculation when $R\to\infty$. When $r$ is arbitrarily small, we consider $D=\mathbb{D}_R^*$
 	and $\gamma=\left\{z:|z|=r,~ 0<r\le R/2\right\}$. Indeed, by the formula \eqref{pncrd diks R}, we have
 	$$
 	\int_\gamma m_D(z)\,|dz|=\frac{4R\pi}{2R-r} \to 2\pi \quad \mbox{as $r\to 0$}.
 	$$
 	This completes the proof.
 	\end{proof}

\noindent\begin{proposition}\label{mD lnth of C_R}
   Let $C_r$ be a circle centered at the origin with radius $r$. Then the following are true:
    \begin{enumerate}
   \item The $m$-perimeter of $C_r$ inside the annulus $A_R$ is given by
    \begin{align*}
    	m_{A_R}\left(C_r\right)=\begin{cases}
    		4\pi\cfrac{R^3r}{(Rr-1)(2R^2-Rr+1)},& \text{if $\cfrac{1}{R}<r\leq \cfrac{1+R^2}{2R}$},\\
    		4\pi\cfrac{Rr}{R^2-r^2},& \text{if $\cfrac{1+R^2}{2R}\leq r<R$}.
    	\end{cases}
    \end{align*}
    \item The $m$-perimeter of $C_r$ inside the punctured disk $\mathbb{D}_R^*$ is given by
    \begin{align*}
    	m_{\mathbb{D}_R^*}\left(C_r\right)=\begin{cases}
    		4\pi\cfrac{Rr}{r(2R-r)},  &\text{if $0<r\leq \cfrac{R}{2}$},\\
    		4\pi\cfrac{Rr}{R^2-r^2},  &\text{if $\cfrac{R}{2}\leq r<R$}.
    	\end{cases}
    \end{align*}
\item For the case $\mathbb{C}^*=A_{\infty}=\mathbb{D}_{\infty}^*$, we have $m_{\mathbb{C}^*}\left(C_r\right)=k_{\mathbb{C}^*}\left(C_r\right)=2\pi$, for any $r$.
\end{enumerate}
\end{proposition}
\begin{proof}
	Both the results follow by considering the parametrization of the circle $|z|=r$ as $z=r\,e^{i\theta}$, $0\leq\theta<2\pi$, and substituting it in the respective length formula.
For the case $R\rightarrow\infty$, we consider the circle $C_r$ of radius $r$ satisfying ${1}/{R}<r\leq{(1+R^2)}/{2R}$ in $A_R$ or satisfying $0<r\leq R/2$ in $\mathbb{D}_R^*$, to get the conclusion.
\end{proof}

\begin{corollary}
	For any two points $z_1$ and $z_2$ with $|z_1|=|z_2|=r<R$, we have the following estimates. 
	\begin{enumerate}
		\item If $z_1,z_2\in A_R$, then
\begin{align*}
		m_{A_R}(z_1,z_2)\leq \begin{cases}
			4\pi\cfrac{R^3r}{(Rr-1)(2R^2-Rr+1)},&  \text{if $\cfrac{1}{R}<r\leq \cfrac{1+R^2}{2R},$}\\[4mm]
			4\pi\cfrac{Rr}{R^2-r^2},&  \text{if $\cfrac{1+R^2}{2R}\leq r<R$},
		\end{cases}
\end{align*}
		\item If $z_1,z_2\in \mathbb{D}_R^*$, then
\begin{align*}
	m_{\mathbb{D}_R^*}(z_1,z_2)\leq \begin{cases}
		4\pi\cfrac{Rr}{r(2R-r)},&  \text{if $0<r\leq \cfrac{R}{2},$}\\[4mm]
		4\pi\cfrac{Rr}{R^2-r^2},&  \text{if $\cfrac{R}{2}\leq r<R$}.
	\end{cases}
\end{align*}
\item For the case $\mathbb{C}^*=A_{\infty}=\mathbb{D}_{\infty}^*,$ $$m_{\mathbb{C}^*}(z_1,z_2)=k_{\mathbb{C}^*}(z_1,z_2)\leq 2\pi.$$
\end{enumerate}
	
\end{corollary}
\begin{proof}
	For any two given points $z_1,z_2\in A_R$ (or $\mathbb{D}_R^*$) with $|z_1|=|z_2|=r<R$, we choose the circle $C_r$ such that the two points lie on it. The results are followed by Proposition \ref{mD lnth of C_R}.
	Note that, in the case of $\mathbb{C}^*$, the bound for $m_{\mathbb{C}^*}(z_1,z_2)$ is also evident from the well-known formula (cf. \cite[Section~2]{MO})
		\begin{align*}
m_{\mathbb{C}^*}(z_1,z_2)=		k_{\mathbb{C}^*}(z_1,z_2)=\sqrt{\alpha^2+\log\cfrac{|z_1|}{|z_2|}}
	\end{align*}
	of the quasihyperbolic metric on $\mathbb{C}^*$, where $\alpha=\measuredangle(z_1,0,z_2)\in\left[0,\pi\right]$.
\end{proof}

\section{\bf Concluding remarks}\label{Concluding remarks}
The goal of this manuscript was to introduce a new metric that generalizes both the hyperbolic and the quasihyperbolic metrics in the sense that it agrees with the hyperbolic metric in balls and with the quasihyperbolic metric in any unbounded domain. Moreover, the new metric is bi-Lipschitz equivalent to the quasihyperbolic metric in arbitrary bounded domains and to the hyperbolic metric in simply connected planar domains. It would be interesting to know if there exists a metric different from the metric $m_D$ with the property that it agrees with the quasihyperbolic metric in unbounded domains and in bounded domains, they are bi-Lipschitz equivalent.

In Section \ref{qsi-inv}, the quasi-invariance properties of the $m_D$-metric under M\"{o}bius, conformal, and quasiconformal maps are described; however, we believe that the constants for each case can be improved. Also, if we go one step further, another interesting problem could be to characterize isometries of the $m_D$-metric. Intuitively, for some cases (e.g. balls), isometries are the M\"{o}bius transformations and for the unbounded domains they are exactly isometries of the quasihyperbolic metric (cf. \cite{Ha07}).

We further investigated the existence of $m_D$-geodesics; however, their uniqueness remains an open question. Since geodesics are solutions to certain differential equations, their local uniqueness is generally assured. Nevertheless, can one also establish global uniqueness, either in arbitrary domains or at least within particular classes of domains? For related results concerning the quasihyperbolic metric, we refer the reader to the work of Rasila and Talponen \cite{RT}.

\section*{\bf Acknowledgements}
The work of the second author is supported by the University Grants Commission (UGC) with Ref. No.: 221610091493. 

\section*{\bf Declaration}
\noindent
{\bf Conflict of interest:} The authors declare that there is no conflict of interest regarding the publication of this article.

\end{document}